\documentclass[12pt]{amsart}
\usepackage{amssymb,mathrsfs,amsthm}
\usepackage{mathabx}
\usepackage{color}
\usepackage{graphicx}
\graphicspath{{./}}
\usepackage{float}
\usepackage{tikz}
\usetikzlibrary{cd}
\usepackage[bottom]{footmisc}
\usepackage[colorlinks,hyperindex,linkcolor=red]{hyperref}
\usepackage{enumerate}
\usepackage{soul}

\newcommand{\cF}{{\mathcal{F}}}

\newcommand{\CP}{{\mathbb{CP}}}
\newcommand{\R}{{\mathbb{R}}}

\newcommand{\F}{{\mathcal{F}}}
\newcommand{\N}{{\mathbb{N}}}
\newcommand{\D}{{\mathbb{D}}}
\renewcommand{\S}{{\mathbb{S}}}
\newcommand{\Z}{{\mathbb{Z}}}

\newcommand{\Op}{{\mathcal{O}p}}

\newcommand{\id}{{\operatorname{id}}}

\renewcommand{\d}{{\operatorname{d}}}

\newcommand{\wtd}{\widetilde}
\newcommand{\Crit}{{\operatorname{Crit}}}

\newcommand{\rot}{{\operatorname{rot}}}

\newtheorem{lemma}{Lemma}
\newtheorem{proposition}[lemma]{Proposition}
\newtheorem{theorem}[lemma]{Theorem}

\newtheorem{corollary}[lemma]{Corollary}
\newtheorem{definition}[lemma]{Definition}
\newtheorem*{theorem*}{Theorem}
\newtheorem*{maintheorem*}{Main Theorem}
\newtheorem*{question*}{Question}
\newtheorem*{proposition*}{Proposition}

\theoremstyle{remark}

\title{An elementary proof of Novikov's Theorem}
\author{Samuel Ranz}
\author{Lauran Toussaint}

\begin{document}

\begin{abstract}
Novikov's theorem states that, given a taut (codimension-one) foliation on a closed $3$-manifold $M$, the fundamental group of any leaf injects into the fundamental group of $M$.
We use foliated branched covers to give a simple proof of this result.
\end{abstract}

\maketitle

\section{Introduction}

Codimension-one foliations\footnote{Unless explicitly stated otherwise, we will always assume our foliations to be smooth, coorientable, and codimension-one.} are extremely flexible and, apart from the Euler characteristic, do not restrict the topology of the manifolds on which they live \cite{Thurston,Eynard-bontemps,Meigniez}. Therefore, it is interesting to consider special classes of foliations. 
One such class consists of taut foliations, and some of their topological properties are described by Novikov's theorem \cite{Novikov}.
In this note, we provide a new proof of this classical result. 

Let us start by recalling the definition:
\begin{definition}\label{def:TautFoliation}
A codimension-one foliation $\cF$ on a closed, orientable, $3$-manifold $M$ is \textbf{taut} if any of the following equivalent conditions hold:
\begin{enumerate}[(i)]
\item There exists a transverse loop intersecting every leaf of $\F$.
\item For any leaf $L$ of $\F$ there exists a closed transversal intersecting $L$.
\item\label{taut3} There exists a closed form $\omega \in \Omega^2(M)$ whose restriction $\omega|_\cF$ is non-degenerate. 
\item There is a Riemannian metric on $M$ for which the leaves of $\F$ are minimal surfaces.
\end{enumerate}
\end{definition}
For a detailed discussion of this definition we refer the reader to \cite{CalegariBook,CandelConlon}.
Taut foliations closely interact with the topology of the ambient manifold, as the following theorem shows.
\begin{theorem*}[\cite{Novikov}]\label{thm:Novikov}
Let $\F$ be a taut foliation on a closed, orientable $3$-manifold $M$. Then for any leaf $L$ the inclusion $\iota:L \to M$ induces an injection
\[ \iota_*:\pi_1(L) \to \pi_1(M).\]
\end{theorem*}

Novikov's original proof \cite{Novikov}, although elementary, is rather delicate. The most recent proofs rely on the theory of minimal surfaces \cite{Rosenberg,Sullivan1,Sullivan2,Hass}, see \cite{CalegariBook} for an excellent exposition.

In this note we reprove this result using another characterization of taut foliations. In short, a foliated branched cover is a smooth map $\pi:(M,\F) \to \CP^1$ whose restriction to any leaf of $\F$ defines a branched cover, see Definition \ref{def:LefschetzPencil} for a more detailed description. 
The fibers of $\pi$ are closed transversals to $\F$. Hence any foliation admitting a branched cover is taut in the sense of Definition \ref{def:TautFoliation}.
The converse can be proven in several ways. An analytical proof is given in \cite[Corollary 1.3]{IbortTorres} while the main theorem of \cite{Calegari} provides a combinatorial proof.

\begin{theorem*}[\cite{Calegari}]
A foliation $\cF$ on a closed $3$-manifold is taut if and only if it admits a foliated branched cover. 
\end{theorem*}

In light of this characterisation, Novikov's theorem is a consequence of the following, which is our main result:

\begin{maintheorem*}\label{thm:maintheorem}
Let $\F$ be a foliation on a closed, orientable $3$-manifold $M$ which admits a foliated branched cover. Then, for any leaf $L$ the inclusion $\iota:L \to M$ induces an injection
\[ \iota_*:\pi_1(L) \to \pi_1(M).\]
\end{maintheorem*}

The proof consists of two steps. First, we show that a foliation admitting a foliated branched cover does not have any vanishing cycles, see Proposition \ref{prop:VanishingCycles}. 
We then follow a standard argument to show that for such foliations the fundamental group of a leaf injects in the fundamental group of the ambient space, see Section \ref{sec:MainTheorem}.

\subsection*{Organization of the paper}
The first four sections contain the required preliminaries. In Section \ref{sec:Ascoli} we recall Ascoli's theorem. Section \ref{sec:LefschetzPencil} contains the basics on foliated branched covers, including Proposition \ref{prop:VanishingCycles} which is the main ingredient for the proof of Theorem \ref{thm:maintheorem}. Section \ref{sec:MorseMaps} contains the definition of Morse maps into foliations and some properties of their characteristic foliations. 
The proof of the main theorem, is then given in Section \ref{sec:MainTheorem}.

\subsection*{Acknowledgements}

We are deeply indebted to Fran Presas for suggesting the strategy of the proof, his insight and ideas, and for many useful discussions. We are grateful to Álvaro del Pino for his comments on a preliminary draft. We also wish to thank Ga\"{e}l Meigniez for useful discussions. \\
The first author is supported by grant BES-2017-081980, project SEV-2015-0554-17-2, of the Spanish Ministry for Science and Innovation.
The second author is funded by the Dutch Research Council (NWO) on the project ``proper Fredholm homotopy theory'' (with project number OCENW.M20.195) of the research programme Open Competition ENW M20-3.

\section{Ascoli's theorem}\label{sec:Ascoli}

We recall here the well-known Ascoli's theorem. 
Its main use will be to show that foliated branched covers do not admit vanishing cycles, see Proposition \ref{prop:VanishingCycles}.

Let $X$ be a topological space and $(Y,d)$ be a metric space, and $S \subset C(X,Y)$ a subset of the space of continuous functions. Then, $S$ is said to be:
\begin{itemize}
    \item \textbf{equicontinuous} at $x_0 \in X$ if, given $\varepsilon > 0$ there exists a neighborhood $U$ of $x_0$ such that
    \[ d(f(x),f(x_0)) < \varepsilon,\]
    for all $f \in S$ and $x \in U$. We say $S$ is equicontinuous if it is equicontinuous at every $x \in X$.
    \item \textbf{pointwise precompact} if for each $x \in X$ the set
    \[ S_x := \{ f(x) \mid f \in S\} \subset Y,\]
    has compact closure.
\end{itemize}

\begin{theorem}[Ascoli's theorem]\label{thm:Ascoli}
Let $X$ be a topological space, $(Y,d)$ a metric space, and consider $C(X,Y)$ endowed with the compact open topology.
If a subset $S \subset C(X,Y)$ is equicontinuous and pointwise precompact, then the closure of $S$ is compact.
\end{theorem}
For a proof of the theorem we refer the reader to \cite{Munkres}.
The following observations and examples will be relevant in the next section:
\begin{itemize}
    \item If $Y$ is compact, then any subset $S \subset C(X,Y)$ is pointwise precompact.
    \item If $Y$ is a compact manifold then any two Riemannian metrics are equivalent. Hence, in this case equicontinuity of a subset $S \subset C(X,Y)$ can be checked with respect to any Riemannian metric on $Y$.
    
    \item Given $S\subset C^\infty(X,Y)$ we denote by 
    \[ J^kS := \{j^k f\mid f \in S\} \subset \Gamma(J^k(X,Y)),\]
    the family of $k$-th order jets of maps in $S$. Suppose that $J^kS$ is equicontinuous and pointwise precompact (with respect to some metric $d$ on $J^k(X,Y)$). Then, by Theorem \ref{thm:Ascoli}, any sequence in $J^kS$ contains a subsequence converging (in $C^0$-norm) to a holonomic section (recall that the space of holonomic sections is closed with respect to the $C^0$-norm, which follows for example from \cite[Theorem 7.17]{Rudin}). 
    Therefore, any sequence in $S$ contains a subsequence which converges in $C^k$-norm.
\end{itemize}
   
By a foliated bundle we mean a fiber bundle $\pi:Y \to B$ endowed with a flat connection, or equivalently, a foliation $\F$ on $Y$ whose leaves are transverse to the fibers of $\pi$. If $Y$ and $B$ have non-empty boundary we additionally require that $\partial Y = \pi^{-1}(\partial B)$. This implies that $\pi$ restricts to a foliated bundle $\pi_\partial:(\partial Y, \F|_{\partial Y}) \to \partial B$.
   
\begin{lemma}\label{lem:EquicontinuityFoliatedBundle}
    Consider a compact Riemannian manifold $(Y,g)$, and $\pi:(Y,\F) \to B$ a foliated bundle. 
    Let $g_s:X \to Y$, $s \in (0,1],$ be a family of leafwise maps all lifting the same map $f:X \to B$. Then, the $g_s$ converge in $C^\infty$-norm to a leafwise map $g_0$ lifting $f$.
    \end{lemma}
    \begin{proof}
        We show that, for any $k \geq 0$, the family $j^kg_s$, $s \in (0,1]$ satisfies the hypothesis of Theorem \ref{thm:Ascoli}.
        Since equicontinuity is a local property it suffices to show it at $x \in X$. Suppose $\dim X = n$. There exist local coordinates $U \simeq \D^n \subset X$ around $x$, $V \simeq \S^1 \times \D^n$ around the fiber $\pi^{-1}(f(x))$, and $W \simeq \D^n \subset B$ around $f(x)$ such that
    \[ \F|_V = \bigcup_{t \in \S^1} \{t\} \times \D^n,\]
    and the restriction $\pi|_V$ equals
    \[ \pi:\S^1 \times \D^n\to \D^n,\quad (t,z) \mapsto z.\]
    We endow $\S^1 \times \D^n$ with the product metric of the standard metrics on $\D^n$ and $\S^1$.
    In these coordinates, each lift $g_s$ equals
    \[ g_s|_U:\D^n \to \S^1 \times \D^n,\quad z \mapsto (t_s,f(z)),\]
    for some $t_s \in \S^1$. From this description it is clear that (using the product metric) for any $k \geq 0$ the family $j^kg_s$, $s \in (0,1]$, is equicontinuous and precompact at $x$. 
    Indeed, the projection of $g_s|_U$ onto the $\D^n$ factor is independent of $s$, while the $t_s$ are contained in a compact set.
    Since $Y$ is compact any two metrics are equivalent, so the same is true with respect to the Riemannian metric. Applying Theorem \ref{thm:Ascoli} we obtain the desired limit.
    \end{proof}

\section{Foliated branched covers}\label{sec:LefschetzPencil}

\begin{definition}\label{def:LefschetzPencil}
A \textbf{foliated branched cover} on a closed, orientable, foliated $3$-manifold is a smooth map
\[ \pi:(M,\F) \to \CP^1,\]
such that $C := \Crit(\pi)\subset M$ is a closed, not necessarily connected, $1$-dimensional submanifold, satisfying the following local model: around each $x \in C$ there exist coordinates $(z,t) \in \D^2 \times \R \subset M$ and a holomorphic chart $\D^2 \subset \CP^1$ in which
\begin{itemize}
    \item $\F = \bigcup_{t \in \R} \D^2 \times \{t\}$;
    \item $\pi(z,t) = \gamma(t) + z^2$, and where $\gamma: \R \to \D^2$ is a curve satisfying $\gamma(0) = \pi(x)$.
\end{itemize}
\end{definition}
In other words, a smooth map $\pi:(M,\F) \to \CP^1$ which restricts to each leaf as a branched cover is a foliated branched cover. 
The existence of a foliated branched cover imposes strong restrictions on the topology of the foliation. In particular, $\F$ does not have any vanishing cycles:
\begin{proposition}\label{prop:VanishingCycles}
Let $\pi:(M,\F) \to \CP^1$ be a foliated branched cover. Suppose we are given a family of immersions of the disk $g_t:D \to (M,\F)$, $t \in (0,1]$, such that:
\begin{itemize}
    \item the image of $g_t$ is contained in a leaf $L_t$ of $\F$;
    \item for every $x \in D$, the curve $g_t(x)$, $t \in (0,1]$, is transverse to $\F$;
    \item as $t$ goes to zero, the restrictions $g_t|_{\partial D} : \partial D \to M$ converge (in the  $C^0$-topology) to an immersion $g_\partial:\partial D \to M$, whose image is contained in a leaf $L$.
\end{itemize}
Then, there exists a map $g:D\to L$ extending $g_\partial$.
\end{proposition}
\begin{proof}

We start by proving the proposition under some extra assumptions, afterwards we will show how to reduce to this special case. So, let us assume for the moment that there exist $p_1,\dots,p_k \in D$ such that for all $t \in (0,1]$:
\begin{itemize}
    \item $\pi \circ g_t = \pi \circ g_1$;
    \item the critical points of $\pi \circ g_t$ are $p_1,\dots,p_k$.
\end{itemize}

For each $1 \leq i \leq k$ we define a curve
\[ \lambda_i:(0,1] \to M,\quad t \mapsto g_t(p_i),\]
into the critical locus of $\pi$.

By compactness there exists a sequence $t_n \to 0$ such that the limit $q_i := \lim_{n \to \infty} \lambda_i(t_n)$ exists.
We fix local coordinates $U_i \simeq \D^2 \subset D$ around $p_i$, $V_i \simeq \D^2 \times [-1,1] \subset M$ around $q_i$, and $W_i \simeq \D^2 \subset \CP^1$ around $f(q_i)$ in which the foliation equals
\[ \F = \bigcup_{t \in [-1,1]} \D^2 \times \{t\}.\]
The second hypothesis implies that in these coordinates the curve equals
\[ \lambda_i:(0,1] \to V_i,\quad t \mapsto (0,a(t)),\]
with $\dot{a}(t) > 0$ for all positive $t$. Moreover, after reparametrizing $t$ we may assume $a(t) = t$.
The foliated branched cover is given by
\[ \pi|_{V_i}:V_i \to W_i,\quad (z,t) \mapsto z^2 + \gamma_i(t),\]
as in Definition \ref{def:LefschetzPencil}.
In particular, this local model implies that $\lim_{t\to0}\lambda_i(t)$ exists and is equal to $q_i$.
Note that our additional assumptions above imply that $\gamma_i(t) = \gamma_i(1)$ for $t \in [0,1]$.
It follows that in these coordinates the family $g_t$ equals:
\[ g_t|_{U_i}:U_i \to V_i,\quad z \mapsto (g_1(z),t).\]
{We endow each $V_i \simeq \D^2 \times [-1,1]$ with a product Riemannian metric and extend them to a Riemannian metric $g$ on the whole of $M$.} The first coordinate of the above expression is independent of $t$. As such it is clear that the family of jets $j^1g_t$, $t \in (0,1]$ is equicontinuous and precompact at $p_i$.

The complement $D \setminus \{p_1,\dots,p_k\}$ is mapped by each $g_t$ into the complement of the critical locus of $f$. Thus, we are in the setup of Lemma \ref{lem:EquicontinuityFoliatedBundle}, and we conclude that $g_t$ and $\d g_t$ are equicontinuous and precompact at every point in $D$. Applying Theorem \ref{thm:Ascoli} we then obtain a $C^1$-limit $g := \lim_{t \to 0} g_t$. By uniqueness of the limit $g$ extends $g_\partial$, and its image is contained in a single leaf since $g$ is a $C^1$-limit of leafwise maps.

It remains to show we can reduce to the special case above.
Let us start by observing that if $\phi_s$ is a leafwise isotopy of $(M,\F)$, then proving the proposition for $g_t$ and $g_\partial$ is equivalent to proving it for $
\phi_1\circ g_t$ and $\phi_1 \circ g_\partial$.
This implies we can assume that the image of $g_\partial$ does not contain any critical points of $\pi$. Indeed, there is a ($C^\infty$-small) isotopy making the image of $g_\partial$ disjoint from the critical points of $\pi$. We will slightly abuse notation and keep denoting the isotoped family by $g_t$ and $g_\partial$.

Next, since we are interested in the limit $\lim_{t\to 0} g_t$, we can restrict the family to $t \in (0,\varepsilon]$ for $\varepsilon> 0$ arbitrarily small. 
Since the image of $g_\partial$ does not contain any critical points of $\pi$ the same is true for $g_t|_{\partial D}$ for $t \in (0,\varepsilon]$. 
Therefore, after restricting the family to $t \in (0,\varepsilon]$, there exists a leafwise isotopy taking $g_t$ to a family for which 
\begin{equation}\label{eq:BoundaryInvariance}
    \pi \circ g_t|_{\partial D} = \pi \circ g_\partial.
\end{equation}
Recall from Definition \ref{def:LefschetzPencil} that the critical points of $\pi$ lie on compact curves transverse to $\F$ and which do not intersect the image of $g_t|_{\partial D}$. 
The family $g_t$ defines a submersion $G:(0,\varepsilon] \times D \to (M,\F)$, such that 
\[ G^*\F = \bigcup_{t \in (0,\varepsilon]} \{t\} \times D.\]
Since the critical curves are transverse to $\F$ their preimages are transverse to the levels $\{t\} \times D$ and do not intersect the boundary.
It follows that the number of critical points $k$ of $\pi \circ g_t$ is independent of $t$, and finite by compactness of $D$.

Hence, after reparametrizing $D$ (for each $t$), we can assume that the critical locus of $\pi \circ g_t$ equals $\{p_1,\dots,p_k\} \subset D$ for all $t$. As before this allows us to define curves
\[ \lambda_i:(0,\varepsilon] \to M,\quad t \mapsto g_t(p_i),\]
into the critical locus of $\pi$, whose limit we denote by $q_i := \lim_{t\to 0} \lambda_i \in M$.
For each $1\leq i \leq k$, we fix local coordinates  in $U_i \simeq \D^2 \subset D$ around $p_i$, $V_i \simeq \D^2 \times [-1,1] \subset M$ around $q_i$, and $W_i \simeq \D^2 \subset \CP^1$ around $\pi(q_i)$ in which the foliation equals
\[ \F = \bigcup_{t \in [-1,1]} \D^2 \times \{t\},\]
and the foliated branched cover is given by:
\[ \pi|_{V_i}:V_i \to W_i,\quad (z,t) \mapsto z^2 + \gamma_i(t).\]
Let $\phi_{i,t}:W_i \to W_i$, $t \in [-1,1]$ be a smooth family of compactly supported diffeomorphisms  satisfying $\phi_{i,t} = \id$ for $t \in \Op(\pm 1)$ and $\phi_{i,t}(\gamma_i(t)) = \gamma_i(\varepsilon)$,
for $t \in \Op([0,\varepsilon])$. We perturb $\pi$ on $V_i$ by setting
\[ \wtd{\pi}|_{V_i}:V_i \to W_i,\quad (z,t) \mapsto (\phi_{i,t}\circ \pi)(z,t).\]
By construction we have that $\wtd{\pi}$ agrees with $\pi$ on the complement of the $V_i$, while on $\wtd{V}_i := \D^2 \times [0,\varepsilon] \subset V_i$ we have:
\begin{equation}\label{eq:tInvariance}
    \wtd{\pi}|_{\wtd{V}_i}: \wtd{V}_i \to W_i,\quad (z,t) \mapsto z^2 + \gamma_i(\varepsilon).
\end{equation}
Again, abusing notation we continue to denote $\wtd{\pi}$ and $\wtd{V}_i$ by $\pi$ and $V_i$ respectively.
Restricting $g_t$ even further if necessary, we can assume that the intersections of the image of $g_t$ with the critical locus of $\pi$ is contained in the union of the $V_i$. Hence, by the above formula, we have 
\[ \pi \circ g_t|_{\{p_1,\dots,p_k\}} = \pi \circ g_\varepsilon|_{\{p_1,\dots,p_k\}}.\]

It remains to arrange that $\pi\circ g_t = \pi \circ g_\varepsilon$ on the rest of $D$. To this end consider the immersion
\[ G:D \times (0,\varepsilon] \to M,\quad (x,t) \mapsto g_t(x),\]
and the subset $B := \{(x,t) \in D \times (0,\varepsilon] \mid x \in \partial D \cup \{p_1,\dots,p_k\}\}$ of its domain.
The fibers of the map $\pi \circ G$ are transverse to the $D$-factor, and at points in $B$ are tangent to the $(0,\varepsilon]$-factor. Hence, there exists a vector field $X \in \mathfrak{X}(D \times (0,\varepsilon])$ tangent to the fibers of $\pi \circ G$ satisfying $\d t(X) = - 1$ and $X|_B = -\partial_t$.
We use its flow $\phi_t$ to define the maps
\[\wtd{g}_t:D \to M,\quad (x,t) \mapsto G(\phi_t(x,\varepsilon)),\quad t \in (0,\varepsilon].\]
By construction these maps satisfy $\pi \circ \wtd{g}_t = \pi \circ \wtd{g}_\varepsilon$ on the whole of $D$.
\end{proof}

\section{Morse maps into foliations.}\label{sec:MorseMaps}

In this section we recall the basic properties of Morse functions into codimension-one foliations, see for example \cite{Moerdijk}.
A smooth function $f:M \to \R$ is called Morse if all of its critical points are non-degenerate. It is well-known that Morse functions form an open and dense subset of $C^\infty(M)$\footnote{Given smooth manifolds $M$ and $N$, we consider $C^\infty(M,N)$ to be equipped with the (strong) $C^\infty$-topology. In this topology a sequence of functions $f_n$, $k \in \N$, converges to $f$ if the derivatives of any order converge uniformly.} see for example \cite{Milnor, Guillemin-Pollack}. 
\begin{lemma}\label{lem:MorseFunctionRelative}
Let $f \in C^\infty(M)$ be a smooth function which is Morse on a neighborhood of a (possibly empty) closed set $A \subset M$. Then there exists a Morse function $g \in C^\infty(M)$ arbitrarily close to $f$ and satisfying $g|_{A} = f|_{A}$.
\end{lemma}

Consider a cooriented codimension-one foliation $\F$ on $N$. Around any point $x \in N$ there exists an open neighborhood $U \subset N$ and a submersion $\phi:U \to \R$ such that
\[ \F = \ker \d \phi,\]
as cooriented distributions.
We refer to $(U,\phi)$ as a \textbf{submersion chart} for $\F$.
Suppose $f:M \to (N,\F)$ is a smooth map.
We say that a point $x \in M$ is a \textbf{singularity of tangency} of $f$ with respect to $\F$, if $\d f(T_xM)$ and $T_{f(x)}\F$ are not transverse as subspaces of $T_{f(x)}N$.
That is, $f$ has a singularity of tangency if and only if $\phi \circ f:f^{-1}(U) \to \R$ has a critical point for some (and hence any) submersion chart $(U,\phi)$ around $f(x)$.
A singularity of tangency $x$ of $f$ is \textbf{non-degenerate} if $\phi_U \circ f$ has a non-degenerate critical point at $x$.
In this case we define the \textbf{index} of $x$ to be the index of $\phi \circ f$.

\begin{definition}
A map $f:M \to (N,\F)$ is called \textbf{Morse} if all of its singularities of tangency are non-degenerate. If $M$ has non-empty boundary we require the singularities of $f$ to be contained in the interior.
\end{definition}

\begin{lemma}\label{lem:FoliatedMorse}
If $M$ is compact the following hold:
\begin{itemize}
\item The set of Morse maps from $M$ to $(N,\F)$ is open and dense in $C^\infty(M,N)$.
\item Suppose $f:M \to (N,\F)$ is a smooth map which is Morse on a neighborhood of a closed subset $A \subset M$. Then there exists a Morse map $g$ arbitrarily close to $f$ such that $g|_A = f|_A$.
\end{itemize}
\end{lemma}
\begin{proof}
Consider a smooth map $f:M \to (N,\F)$. Since $M$ is compact, its image can be covered by finitely many submersion charts $(U_i,\phi_i)$ for $\F$.
The proof follows from inductively applying Lemma \ref{lem:MorseFunctionRelative} to $\phi_i\circ f:f^{-1}(U_i) \to \R$.
\end{proof}

As for Morse functions, the singularities of $f:M \to (N,\F)$ are isolated and satisfy a local model:
Let $\F$ be the foliation by horizontal hyperplanes in $\R^{n+1}$. 
For $1 \leq k \leq n$, the embedding
\begin{equation}\label{eq:MorseLocalModel}
f_k: \R^n \to (\R^{n+1},\F),\quad (x_1,\dots,x_n) \mapsto \Big(x_1,\dots,x_n,\sum_{i=1}^k -x_i^2 + \sum_{i = k+1}^n x_i^2\Big),
\end{equation}
is Morse, and has a index $k$ singularity at the origin. Locally, around a singularity of index $k$, any Morse map is equivalent to this model.

\begin{lemma}\label{lem:EulerCharacteristic}
Let $M$ be compact and $f:M \to (N,\F)$ a Morse map. Then
\[ \chi(M) = \sum_{p \in \operatorname{Sing(f)}} (-1)^{\operatorname{Index}(f,p)}.\]
\end{lemma}
\begin{proof}
Cover the image of $f$ by submersion charts $(U_i,\phi_i)$ for $\F$. This allows us to define the gradient of $\phi_i\circ f$ on $f^{-1}(U_i)$. Using a partition, these local vector fields can be patched together to a vector field $X$ on $M$.
It is easily checked that for any critical point $p$ the index of $f$ and the gradient $X$ are related by
\[ \operatorname{Index}(X,p) = (-1)^{\operatorname{Index}(f,p)}.\]
 As such the result follows from the Poincar\'e-Hopf theorem, see for example \cite{Guillemin-Pollack}.
\end{proof}

We now restrict to $\dim M = 2$ and $\dim N = 3$. In this case the pullback $f^*\F$ defines a singular foliation by lines on $M$, called the \textbf{characteristic foliation} of $f$.
To be precise, a leaf of $f^*\F$ is defined to be the preimage under $f$ of a leaf of $\F$.
On the complement of the critical points of $f$ the characteristic foliation is regular, and the singularities of $f^*\F$ correspond to the critical points of $f$. Their local models are given by Equation \ref{eq:MorseLocalModel}. The singularity corresponding to a critical point of index $0$ or $2$ is called a \textbf{center}, while the singularity corresponding to a critical point of index $1$ is called a \textbf{saddle}, see also Figure \ref{fig:SaddleConfigurations}.

Thus, Lemma \ref{lem:EulerCharacteristic} immediately implies the following:
\begin{corollary}\label{cor:NumberOfSingularities}
Suppose $f:\D^2 \to (M^3,\F)$ is Morse.
Then, the number of centers of the characteristic foliation $f^*\F$ is one more than the number of saddles. In particular, it has at least one center singularity.
\end{corollary}

\section{Proof of the main theorem}\label{sec:MainTheorem}

In light of Proposition \ref{prop:VanishingCycles} it suffices to show that if $(M,\F)$ does not have any vanishing cycles, then for any leaf $L$ there is an injection $\pi_1(L) \to \pi_1(M)$. This is a classic result which can be found in many places in the literature, see for example \cite[Proposition 9.2.5]{CandelConlon}, or \cite[Theorem 4.35]{CalegariBook}, except Lemma \ref{lem:immersedcappingdisk} below which is sometimes not explicitely stated. We include it here for completeness.

Suppose we are given a foliated manifold $(M^3,\F)$ and a leafwise curve $\gamma:S^1 \to L$ which is contractible in $M$. We start by fixing an immersed capping disk $f:\D^2 \to M$ as in the following lemma:

\begin{lemma}\label{lem:immersedcappingdisk}
Consider a foliated manifold $(M^3,\F)$ and $\gamma:\S^1 \to L$ a leafwise curve which is contractible in $M$. Then, there exists an immersion $f:\D^2 \to M$ satisfying:
\begin{enumerate}
\item $f(\partial \D^2) \subset L$, and $f|_{\partial \D^2}$ is leafwise homotopic (inside $L$) to $\gamma$.
\item $f|_{\Op(\partial \D^2)}$ is transverse to $\F$.
\end{enumerate}
\end{lemma}

\begin{proof}
By a $C^0$-small leafwise homotopy we can assume $\gamma$ is immersed. Slightly pushing it off, transverse to $\F$, we obtain an immersed cylinder transverse to $\F$. Since any homotopy of $\gamma$ stays contractible in $M$, we find a smooth map $g:\D^2 \to M$ extending the immersed cylinder. That is, $g|_{\partial \D^2} = \gamma$, and $g|_{\Op \partial \D^2}$ is an immersion transverse to $\F$.

The Smale-Hirsch immersion theorem \cite{Hirsch59}, see also \cite[Theorem 8.2.1]{EliMisch}, implies that to find our desired immersion $f$, it suffices to find an injective bundle map $G:T\D^2 \to TM$ covering $g$, and such that $G$ agrees with $\d g$ at points in $\partial D^2$. 

The differential of $g$ defines a bundle map $\d g: T\D^2 \to g^*TM$ covering the identity. Using that the disk is contractible we fix a framing $e := (e_1,e_2,e_3)$ of $g^*TM$ such that
$\F = \langle e_1,e_2 \rangle$ (implying $e_3 \pitchfork \F$), and 
\[ \d_p g(\partial_\theta) = e_2, \quad \d_p g(\partial_r) = e_3,\quad \forall p \in \partial \D^2,\]
where $(r,\theta)$ denote polar coordinates on $\D^2$.
In the above framing $\d g|_{\partial D^2}$ can be interpreted as family of orthonormal $2$-frames
\[ G_\partial:\S^1 \to V_2(\R^3),\quad t \mapsto (\dot{\gamma}(t),e_3).\]
As such, the existence of $G$ is equivalent to $G_\partial$ defining the zero class in $\pi_1(V_2(\R^3)) = \Z_2$, which in turn is equivalent to the rotation number of $\dot{\gamma}:\S^1 \to \R^2$ being odd. 

Let $U \subset L$ be a contractible neighborhood of $\gamma(0)$, and extend the framing $e$ over $U$. Let $\sigma:\S^1 \to U$ be an immersed curve with $\sigma(0) = \gamma(0)$, $\dot{\sigma}(0) = \dot{\gamma}(0)$, and $\rot(\dot{\sigma}) = 1$. Then, $\gamma$ is leafwise homotopic (though not through immersed curves) to the immersed curve $\gamma \circ \sigma$ which has rotation number $\rot(\gamma \circ \sigma) = \rot(\gamma) + 1$. Thus, we can assume that $\gamma$ has odd rotation number implying that an injective bundle map $G:T\D^2 \to TM$ covering $g$ exists. Applying the Smale-Hirsch immersion theorem relative to the boundary $\partial D^2$ gives the desired immersion $f$.
\end{proof}

The induced characteristic foliation $\wtd{\F} := f^*\F$ is tangent to the boundary $\partial D^2$. By compactness it contains finitely many center and saddle singularities. After slightly perturbing $f$ if necessary, we may assume that each of the leaves of $\wtd{\F}$ contains at most one singularity. 
For each center singularity $c$, let $\Delta$ be the maximal set satisfying:
\begin{itemize}
    \item $\Delta$ contains $c$ and is $\wtd{\F}$ saturated.
    \item Each leaf of $\wtd{\F}$ contained in $\Delta$ bounds an immersed disk in the corresponding leaf of $\F$.
\end{itemize}
The local model of a center singularity shows that $\Delta$ is non-empty, and it follows from local Reeb stability that $\Delta$ is open. Since $\Delta$ is topologically a disk, the boundary $\partial \Delta:= \bar{\Delta} \setminus \mathring{\Delta}$, is a closed (not necessarily immersed) curve contained in a leaf of $\wtd{\F}$.

Now there are several possibilities. First, assume that $\wtd{\F}$ has only one singularity which is necessarily a center. In this case $\partial \Delta$ is contained in a non-singular leaf of $\wtd{\F}$ and hence itself immersed. That is, $\partial \Delta$ is a vanishing cycle giving a contradiction by Proposition \ref{prop:VanishingCycles}.

Next, suppose that $\wtd{\F}$ has more than one center. The previous argument shows that for each center $c$ the corresponding curve $\partial \Delta$ is singular. Since each leaf of $\wtd{\F}$ contains at most one singular point, so does $\partial \Delta$, and this singularity is a saddle. 
By Corollary \ref{cor:NumberOfSingularities} there must be at least two centers $c_1$ and $c_2$ for which the boundaries $\partial \Delta_1$ and $\partial \Delta_2$ share the same saddle, and thus intersect in a point. There are two possible configurations, depicted in Figure \ref{fig:SaddleConfigurations} below.

In the first case, the union $\partial \Delta_1 \cup \partial \Delta_2$ bounds a simply connected domain in the corresponding leaf of $\F$. In the second case, suppose without loss of generality that $\partial \Delta_1 \subset \Delta_2$, and denote the saddle singularity by $s$. Then, the closed curve $\partial \Delta_2 \setminus \partial \Delta_1 \cup s$ bounds a disk in the corresponding leaf of $\F$. Thus, in both cases, $f$ can be replaced by $\wtd{f}$ for which the induced characteristic foliation has one less center singularity. Hence we can inductively reduce to the case that $\wtd{\F}$ has only one center, which we considered above.

\begin{figure}[ht!]
    \centering
    \def\svgwidth{\textwidth}
    \includegraphics[scale=0.88]{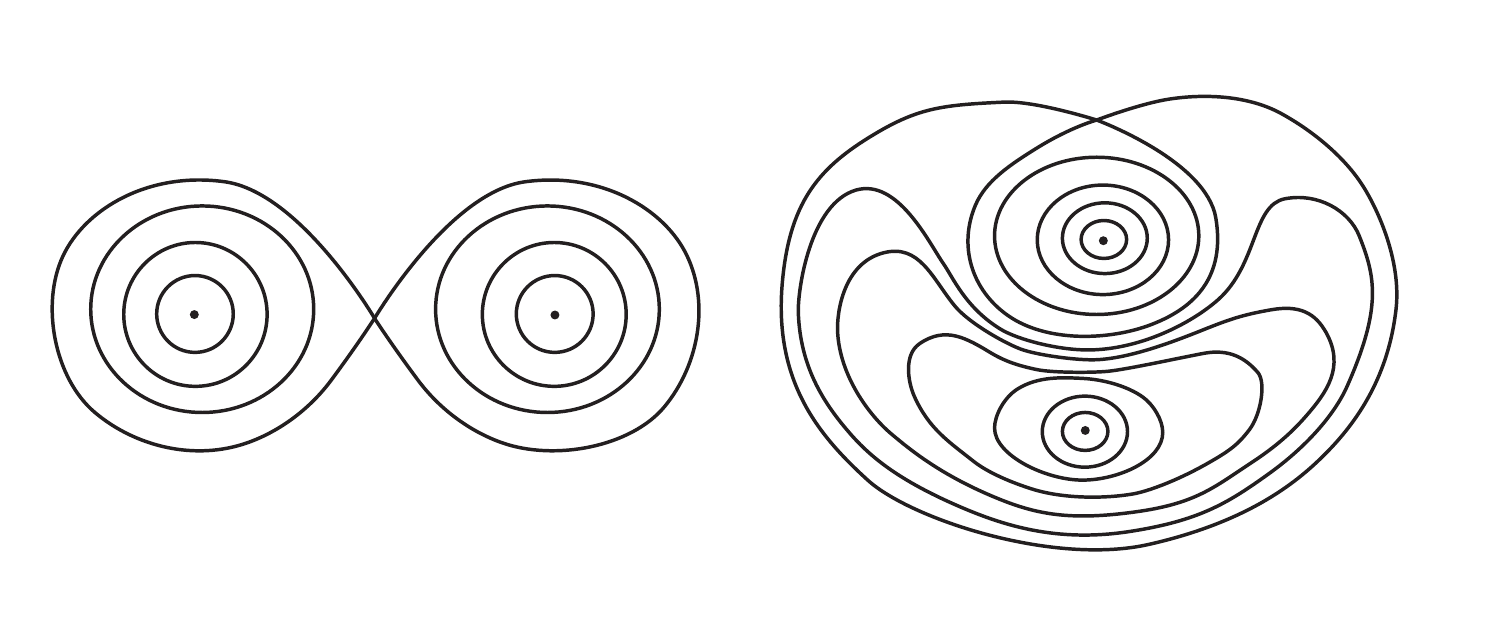}
    \caption{The two possible configurations of a saddle singularity together with its corresponding center singularities.}
    \label{fig:SaddleConfigurations}
\end{figure}

\newpage

\bibliographystyle{abbrv}
\bibliography{Bibliography}

\begin{thebibliography}{10}

\bibitem{CalegariBook}
D.~Calegari.
\newblock {\em Foliations and the geometry of 3-manifolds}.
\newblock Oxford Mathematical Monographs. Oxford University Press, Oxford,
  2007.

\bibitem{Calegari}
D.~Calegari.
\newblock Taut foliations leafwise branch cover {$S^2$}.
\newblock {\em Algebr. Geom. Topol.}, 21(5):2523--2541, 2021.

\bibitem{CandelConlon}
A.~Candel and L.~Conlon.
\newblock {\em Foliations. {II}}, volume~23 of {\em Graduate Studies in
  Mathematics}.
\newblock American Mathematical Society, Providence, RI, 2000.

\bibitem{EliMisch}
Y.~Eliashberg and N.~Mishachev.
\newblock {\em Introduction to the {$h$}-principle}, volume~48 of {\em Graduate
  Studies in Mathematics}.
\newblock American Mathematical Society, Providence, RI, 2002.

\bibitem{Eynard-bontemps}
H.~Eynard-Bontemps.
\newblock On the connectedness of the space of codimension one foliations on a
  closed 3-manifold.
\newblock {\em Invent. Math.}, 204(2):605--670, 2016.

\bibitem{Guillemin-Pollack}
V.~Guillemin and A.~Pollack.
\newblock {\em Differential topology}.
\newblock Prentice-Hall, Inc., Englewood Cliffs, N.J., 1974.

\bibitem{Hass}
J.~Hass.
\newblock Minimal surfaces in foliated manifolds.
\newblock {\em Comment. Math. Helv.}, 61(1):1--32, 1986.

\bibitem{Hirsch59}
M.~W. Hirsch.
\newblock Immersions of manifolds.
\newblock {\em Trans. Amer. Math. Soc.}, 93:242--276, 1959.

\bibitem{IbortTorres}
A.~Ibort and D.~Mart\'{\i}nez~Torres.
\newblock Lefschetz pencil structures for 2-calibrated manifolds.
\newblock {\em C. R. Math. Acad. Sci. Paris}, 339(3):215--218, 2004.

\bibitem{Meigniez}
G.~Meigniez.
\newblock Regularization and minimization of codimension-one {H}aefliger
  structures.
\newblock {\em J. Differential Geom.}, 107(1):157--202, 2017.

\bibitem{Milnor}
J.~Milnor.
\newblock {\em Morse theory}.
\newblock Annals of Mathematics Studies, No. 51. Princeton University Press,
  Princeton, N.J., 1963.
\newblock Based on lecture notes by M. Spivak and R. Wells.

\bibitem{Moerdijk}
I.~Moerdijk and J.~Mrcun.
\newblock {\em Introduction to Foliations and Lie Groupoids}.
\newblock Cambridge Studies in Advanced Mathematics. Cambridge University
  Press, 2003.

\bibitem{Munkres}
J.~R. Munkres.
\newblock {\em Topology}.
\newblock Prentice Hall, Inc., Upper Saddle River, NJ, 2000.
\newblock Second edition of [ MR0464128].

\bibitem{Novikov}
S.~P. Novikov.
\newblock The topology of foliations.
\newblock {\em Trudy Moskov. Mat. Ob\v{s}\v{c}.}, 14:248--278, 1965.

\bibitem{Rosenberg}
H.~Rosenberg.
\newblock Foliations by planes.
\newblock {\em Topology}, 7:131--138, 1968.

\bibitem{Rudin}
W.~Rudin.
\newblock {\em Principles of mathematical analysis}.
\newblock International Series in Pure and Applied Mathematics. McGraw-Hill
  Book Co., New York-Auckland-D\"{u}sseldorf, third edition, 1976.

\bibitem{Sullivan1}
D.~Sullivan.
\newblock Cycles for the dynamical study of foliated manifolds and complex
  manifolds.
\newblock {\em Invent. Math.}, 36:225--255, 1976.

\bibitem{Sullivan2}
D.~Sullivan.
\newblock A homological characterization of foliations consisting of minimal
  surfaces.
\newblock {\em Comment. Math. Helv.}, 54(2):218--223, 1979.

\bibitem{Thurston}
W.~P. Thurston.
\newblock Existence of codimension-one foliations.
\newblock {\em Ann. of Math. (2)}, 104(2):249--268, 1976.

\end{thebibliography}
\end{document}